\documentclass[11pt]{amsart}
 \pdfoutput=1 

\usepackage{amsmath}
\usepackage{amsfonts}   
\usepackage{amsthm}     

\usepackage{enumitem}
\setlist[enumerate,1]{label=\alph*),ref=\theenumi}

\newtheorem{lemma}{Lemma}[section]
\newtheorem{proposition}[lemma]{Proposition}
\newtheorem{theorem}[lemma]{Theorem}
\newtheorem{example}[lemma]{Example}
\newtheorem{remark}[lemma]{Remark}
\usepackage{todonotes}
\newcommand{\R}{\mathbb{R}}      
\newcommand{\N}{\mathbb{N}}                        

\newcommand{\Lcal}{\mathcal{L}}                         
\newcommand{\Jcal}{\mathcal{J}}                         

\newcommand{\embed}{\hookrightarrow}

\newcommand{\abs}[1]{|#1|}
\DeclareMathOperator{\LspaceSymbol}{\mathbf{L}} 
\newcommand{\Lpspace}[1][p]{\LspaceSymbol^{{#1}}}
\newcommand{\LGspace}[1][\DefaultGfunction]{\LspaceSymbol^{{#1}}}
\newcommand{\LGastspace}[1][\DefaultGfunction]{\LspaceSymbol^{{#1^\star}}}

\DeclareMathOperator{\WspaceSymbol}{\mathbf{W}} 

\newcommand{\WLGspace}[1][\DefaultGfunction] { {\WspaceSymbol^1}\LspaceSymbol^{{#1}} }

\newcommand{\norm}[1]{\|#1\|}                           
\newcommand{\inner}[2]{\left\langle #1,#2\right\rangle} 

\newcommand{\LGnorm}[2][\DefaultGfunction]{\norm{#2}_{\LGspace[{#1}]}}
\newcommand{\LGastnorm}[2][{\DefaultGfunction}]{\norm{#2}_{\LGastspace[{#1}]}}
\newcommand{\WLGnorm}[2][\DefaultGfunction]{\norm{#2}_{\WLGspace[{#1}]}}

\newcommand{\WLGastnorm}[2][\DefaultGfunction]{\norm{#2}_{(\WLGspace[{#1}])^{\star}}}
\newcommand{\WLGTspace}[1][\DefaultGfunction]{ {\WspaceSymbol^1_T}\LspaceSymbol^{{#1}}}

\newcounter{Fcounter}
\newcounter{Vcounter}
\begin{document}


\title[Two periodic solutions to general anisotropic operator]{Existence of two periodic solutions to general anisotropic Euler-Lagrange equations }

\begin{abstract}
This paper is concerned with the following Euler-Lagrange system
\begin{equation*}
\begin{cases}
\frac{d}{dt}\Lcal_v(t,u(t),\dot u(t))=\Lcal_x(t,u(t),\dot u(t))\quad \text{ for a.e. }t\in[-T,T],\\
u(-T)=u(T),
\end{cases}
\end{equation*} where Lagrangian is given by $\Lcal=F(t,x,v)+V(t,x)+\langle f(t), x\rangle$, growth conditions are determined by an anisotropic G-function and some geometric conditions at infinity. We consider two cases: with and without forcing term $f$.
Using a general version of the  Mountain Pass Theorem  and Ekeland's variational principle we prove the existence of at least two nontrivial periodic solutions in an anisotropic Orlicz-Sobolev space.
\end{abstract}

\author{M. Chmara}

\address{
Department of Technical Physics and Applied Mathematics,
Gda\'{n}sk University of Technology,
Narutowicza 11/12, 80-233 Gda\'{n}sk, Poland
}
\email{magdalena.chmara@pg.edu.pl}


\maketitle


\section{Introduction}

We consider the  second order system:
\begin{equation}
\tag{AELT}
\label{eq:AELT}
\begin{cases}
\frac{d}{dt}\Lcal_v(t,u(t),\dot u(t))=\Lcal_x(t,u(t),\dot u(t))\quad \text{ for a.e. }t\in I,\\
u(-T)=u(T),
\end{cases}
\end{equation}
where $I=[-T,T]$, $|I|\geq1$ and $\Lcal\colon I\times\R^N\times\R^N\to \R$ is given by
\begin{equation*}
\Lcal(t,x,v)=F(t,x,v)+V(t,x)+\inner{f(t)}{x}.
\end{equation*}

The growth of $\Lcal$ is determined by function $G$ such that:
\begin{enumerate}[ref=(G),label=(G)]
	\item \label{asm:G}
	$G\colon \R^N\to [0,\infty)$ is a continuously differentiable G-function (i.e. $G$ is convex, even, $G(0)=0$ and $G(x)/|x|\to \infty$ as $|x|\to \infty$) satisfying $\Delta_2$ and $\nabla_2$ conditions (at infinity).
\end{enumerate}
In section \ref{sec:preliminaries} we use $G$ to define an Orlicz space.
The following assumption on $F$, $V$ and $f$ will be needed throughout the paper.

Function $F\colon I\times\R^N\times\R^N\to \R$ is of class $C^1$ and satisfy
\begin{enumerate}[label=(F$_\arabic*$),ref=(F$_\arabic*$)]
	\item
	\label{asm:F:convex}
	$F(t,x,\cdot)$ is convex for all $(t,x)\in I\times \R^N$,
	\item
	\label{asm:F:growth}
	there exist $a\in C(\R_+,\R_+)$ and $b\in\Lpspace[1](I,\R_+)$ such that for all $(t,x,v)\in I\times \R^N\times\R^N$:
	\begin{gather}
	\label{F:growth:F}
	|F(t,x,v)|\leq a(|x|)\,(b(t)+G(v)),\\
	\label{F:growth:Fx}
	|F_x(t,x,v)|\leq a(|x|)\,(b(t)+G(v)), \\
	\label{F:growth:Fv}
	G^{\ast}(F_v(t,x,v))\leq a(|x|)\,(b(t)+\,G^{\ast}\left(\nabla G(v)\right)),
	\end{gather}
	\item \label{asm:F:AR}
	there exists $\theta_{F} > 0 $ such that for all $(t,x,v)\in I\times \R^N\times\R^N$:
	\begin{gather*}
	\inner{F_x(t,x,v)}{x} + \inner{F_v(t,x,v)}{v} \leq \theta_F\, F(t,x,v),
	\end{gather*}
	\item \label{asm:F:ellipticity}
	there exists $\Lambda>0$ such that for all $(t,x,v)\in I\times \R^N\times\R^N$:
	\begin{equation*}
	F(t,x,v) \geq \Lambda\, G(v),
	\end{equation*}
	\item \label{asm:F:zero}
	$F(t,x,0) = 0$, $F_v(t,x,0) = 0$ for all $(t,x)\in I\times \R^N$,
	\setcounter{Fcounter}{\value{enumi}}
\end{enumerate}

On potential $V\in C^1( I\times\R^N, \R)$ we assume:

\begin{enumerate}[label=$(V_{\arabic*})$]
	\item\label{V:K-W}
	$V(t,x)=K(t,x)-W(t,x)$ 
	for $x\in\R^N$, $t\in I$,
	\item \label{V:ARad} there exist $M>0$, $\varepsilon_V>0$, $\theta_V>\theta_F$, $1<p_K\leq \theta_V-\varepsilon_V$, such that
	\begin{equation}
	\label{V:AR}
	\inner{V_x(t,x)}{x}\leq (\theta_V-\varepsilon_V) K(t,x)-\theta_V	W(t,x)\quad
	\text{ and}
	\end{equation}
	\begin{equation}
	\label{V:W>K>pk}
	W(t,x)>K(t,x)>|x|^{p_K}
	\end{equation}
	for $|x|>M$, $t\in I$,
	\item \label{V:>G}
	there exist $\rho$, $b>1$ and $g\in\LGspace[1](I,\R)$, such that $V(t,x)\geq b\, G(x)-g(t) $ for 
	$G(x/(2|I|))\leq\rho/2$,
	\item\label{V:zero} $\int_I V(t,0)\,dt=0$ for $t\in I$. 
	\setcounter{Vcounter}{\value{enumi}}
\end{enumerate}
We assume that the forcing term $f$ belongs to the space $\LGastspace$ and
\begin{enumerate}[label=(f)]
	\item\label {f:IntG*f<}
	$\int_IG^{\ast}(f(t))+g(t)\,dt<\min\{\Lambda,b-1\}\rho$.
\end{enumerate}

Using the general form of the Mountain Pass Theorem  and Ekeland's Variational Principle we show that the problem \eqref{eq:AELT} has at least two nontrivial solutions in an anisotropic Orlicz-Sobolev space (theorems \ref{thm:main1}, \ref{thm:main3} and \ref{thm:main2}).

Problems similar to \eqref{eq:AELT} were considered e.g. in \cite{Dao16,Ter12,ZhaYua17} for $F(t,x,v)=|v|^p$, in \cite{AciMaz17} for $F$ being an isotropic G-function and in \cite{AciMaz19,ChmMak19} (periodic problem) for $F$ being an anisotropic G-function.
Recently  in \cite{ChmMak20} authors proved the existence of a Dirichlet problem, where $F$ is in general form and satisfies \ref{asm:F:convex}-\ref{asm:F:zero}.

In \cite{Dao16,Ter12,ZhaYua17,ChmMak19,ChmMak20} the existence of a mountain pass type solution was shown using the well known Mountain Pass Theorem by Ambrosetti and Rabinowitz \cite{AmbRab73}. 
One of the assumptions of this theorem is that the action functional is greater than zero on the boundary of some ball,
i.e.
\begin{equation}
\label{asm:mpt_ar_J>0}
\text{there exists $r>0$ such that if $\norm{x}= r$, then $J(x)>0$}.
\end{equation}   
In the Orlicz-Sobolev space norm is a sum of Luxemburg norms $\WLGnorm{u}=\LGnorm{\dot u}+\LGnorm{u}$ (see more in section \ref{sec:preliminaries}).
To apply theorem of Ambrosetti and Rabinowitz to the anisotropic Orlicz-Sobolev space setting it was necessary 
to establish connections between Luxemburg norm $\LGnorm{\cdot}$ and modular $\int_IG(\cdot)\,dt$ (see papers \cite{ChmMak19,ChmMak20,BarCia17}).
%
%

It turns out that the ball $B_r=\{u\in\WLGspace\colon \WLGnorm{u}<r\}$ is not the most suitable set to obtain the mountain pass geometry  in the anisotropic case (an example explaining this fact can be found in section \ref{sec:example}).
Therefore, instead of the ball, we use set 
\begin{equation}
\label{omega}
\Omega=\Phi^{-1}([0,\rho))=\left\{u\in\WLGspace: \int_I G(\dot u)+G(u)\,dt<\rho _0 \right\},
\end{equation}
where 
$\Phi:\WLGspace\to[0,\infty)$ is given by 
$\Phi(u)=\int_IG(\dot u)+G(u)\,dt.
$ $\Phi$ is not a norm, but it is better suited to geometric idea of MPT in the anisotropic case.

In the literature we can find a lot of versions of the Mountain Pass Theorem.
In our case we use general form of the MPT with a bounded open neighborhood instead of a ball. 
The following theorem is a direct consequence of theorem 4.10 from \cite{MawWil89}.

\begin{theorem}
	\label{thm:MPT_mw}
	Let $X$ be a Banach space and $J\in C^1(X,\R)$. Assume that
	there exist $e_0, e_1\in X$ and a bounded open neighborhood $\Omega$ of $e_0$ such that $e_1\in X\backslash\Omega$ and
	\begin{equation}
	\label{cond:geom_mpt}
	\inf_{u\in\partial\Omega}J(u)>\max\{J(e_0),J(e_1)\}.
	\end{equation} 	 
	Let \[\Gamma=\{g\in C([0,1],X):~ g(0)=e_0,~ g(1)=e_1\}\]
	and
	\[
	c=\inf_{g\in\Gamma}\max_{s\in[0,1]}J(g(s)),
	\]
	If $J$ satisfies the Palais-Smale condition, then $c$ is a critical value of $J$ and \[c>\max\{J(e_0),J(e_1)\}.\] 
\end{theorem}

	In fact, there is only one difference between theorem \ref{thm:MPT_mw} and theorem 4.10 from \cite{MawWil89}. In \cite{MawWil89} it is assumed that $J$ satisfies the Palais Smale condition at the level $c$. When the Palais-Smale condition is satisfied (i.e.
	$\{J(u_n)\}$ is bounded and  $J'(u_n ) \to 0$)
	, we can check immediately that the Palais Smale condition at the level $c$
	(i.e.
	$\{J(u_n)\}\to c$ and $J'(u_n ) \to 0$)
	holds for all $c\in\R$ \cite[p. 16]{Jab03}.

In the proof of the first main theorem \ref{thm:main1} we show that our action functional satisfies conditions of theorem \ref{thm:MPT_mw} where $\Omega$ is given by \eqref{omega}.

Using the Ekeland Variational Principle we show the existence of a second nontrivial solution of \eqref{eq:AELT}, which belongs to the interior of $\Omega$.
The existence of two solutions was investigated for example in \cite{ZhaYua17} for ordinary $p$-Laplacian systems and in \cite{YuaZha16}  for
$p(t)$-laplacian systems.
In these papers it was shown that there exists $u_2$  such that
\begin{equation*}
\label{eq:J<0}
\Jcal(u_2)=\inf_{u\in\overline {B_{r}}}\Jcal(u)\leq0 
\quad \left(\text{or } <0, \text{ in \cite{ZhaYua17}}\right)
\end{equation*}
and that there exists a minimizing sequence $\{v_n\}$ which is a Palais-Smale sequence of $\Jcal$,
contained
in a small ball centered at 0. In our case we use similar methods, but instead of the ball  $B_{r}$ we take $\Omega=\Phi^{-1}([0,\rho))$. Since $\Omega$ is not a ball we cannot simply cite \cite{ZhaYua17,YuaZha16}. Our proof is based on concepts of \cite{MawWil89} and \cite{AmbHaj18}.

We shall distinguish cases with and without forcing.
In the case with forcing it is enough that 
\begin{equation*}
\inf_{\overline {u}\in\Omega}\Jcal(u)\leq 0.
\end{equation*}
In the case without forcing $u_0\equiv 0$ is a trivial solution of \eqref{eq:AELT}, so
we need the sharp inequality.  
To obtain this we need additional assumptions:
\begin{enumerate}[label=(F$_\arabic*$),ref=(F$_\arabic*$)]
	\setcounter{enumi}{\value{Fcounter}}
	\item \label{asm:F:nabla2type}	there exist $\zeta_F>1$, $\lambda_0\in(0,1)$ such that
	\begin{equation*}
	F(t,\lambda x,\lambda v)\leq \lambda^{\zeta_F}F(t,x,v)
	\end{equation*}
	for $\lambda\in(0,\lambda_0), t\in I$, $G\left(\frac{x}{2|I|}\right)\leq\rho/2$,
\end{enumerate}

\begin{enumerate}[label=(V$_\arabic*$),ref=(V$_\arabic*$)]
	\setcounter{enumi}{\value{Vcounter}}
	\item \label{V:nabla2type}
	there exist $\zeta_K, \zeta_W>1$,  $\zeta_W<\min\{\zeta_F,\zeta_K\},$ $\lambda_0\in(0,1)$ such that $W(t,x)>0$ and
	\begin{equation*}
	V(t,\lambda x)\leq \lambda^{\zeta_K}K(t,x)-\lambda^{\zeta_W}W(t,x)
	\end{equation*}
	for $\lambda\in(0,\lambda_0), t\in I$, $G\left(\frac{x}{2|I|}\right)\leq\rho/2$.
\end{enumerate}

\section{Preliminaries}
\label{sec:preliminaries}

A function $G^\ast(y)=\sup_{x\in\R^N}\{\inner{x}{y}-G(x)\}$ is called the Fenchel conjugate of
$G$. As an immediate consequence of the definition we have the Fenchel inequality:
\begin{equation*}
\label{ineq:Fenchel}
\forall_{x,y\in\R^N}\ \inner{x}{y}\leq G(x)+G^{\ast}(y).
\end{equation*}

Now we briefly recall the notion of anisotropic Orlicz spaces. For more details we refer the reader to
\cite{Sch05} and \cite{ChmMak17}. The Orlicz space associated with $G$ is defined to be
\[
\LGspace =\LGspace(I,\R^N)= \{u\colon I\to \R^N: \int_{I} G(u) \,dt<\infty\}.
\]
The space $\LGspace$ equipped with the Luxemburg norm
\[
\LGnorm{u}=\inf\left\{\lambda>0: \int_{I} G\left(\frac{u}{\lambda}\right)\, dt\leq 1 \right\}
\]
is a reflexive Banach space. We have the H\"older inequality
\[
\int_I \inner{u}{v}\,dt\leq 2\LGnorm{u}\LGnorm[G^\ast]{v}
\quad\text{ for every $u\in \LGspace$ and $v\in \LGspace[G^\ast]$.}
\]
Let us denote by
\[
\WLGspace=\WLGspace(I,\R^N)=\left\{u\in \LGspace\colon \dot{u}\in \LGspace \right\}
\]
an anisotropic Orlicz-Sobolev space of vector valued functions with the usual norm
\[
\WLGnorm{u}=\LGnorm{u}+\LGnorm{\dot u}.
\]
It is known that elements of $\WLGspace$ are absolutely continuous functions.

We introduce the following subset of $\WLGspace$ 
\begin{equation*}
\WLGTspace=\left\{u\in \WLGspace : ~	u(-T)=u(T)\right\}.
\end{equation*}

Functional $R_G:\LGspace\to[0,\infty)$ given by formula
$R_G(u):=\int_IG(u)\,dt$ is called modular (see
\cite{KhaKoz15} and \cite{Mus83}). 
For Lebesgue spaces  notions of modular and norm are indistinguishable because
\[
\LGnorm[p]{u}=\left(\int_{I}|u(t)|^p\,dt\right)^{1/p}
\] 

In the Orlicz space this relation is more complicated.
One can easily see, that 
\begin{equation}
\label{eq:modular_norm_relation}
R_G(u) \leq
\LGnorm{u} \text{ for }
\LGnorm{u}\leq 1  \text{ and } R_G(u) > \LGnorm{u} \text{ for }
\LGnorm{u}> 1.
\end{equation}

The modular $R_G$ is coercive in the following sense \cite[prop. 2.7]{Le14}:
\begin{equation} 
\label{prop:RGcoercive}
\lim_{\LGnorm{u}\to\infty}\frac{R_G(u)}{\LGnorm{u}}=\infty.
\end{equation}
%

The following lemma is crucial to theorems   \ref{thm:main1} and \ref{thm:main3}.
We will use it to show that $\Jcal\rvert_{\partial\Omega}>0$  and that $\Jcal$ is negative in the interior of $\Omega$. 
\begin{lemma}
	\label{lem:RG>G}
	If $|I|\geq 1$ then $R_G(\dot u)+R_G(u)\geq 2 G\left(\frac{u(t)}{2\abs{I}}\right)$ for $t\in I$, $u\in\WLGspace$.
\end{lemma}
\begin{proof}
	Let $u\in\WLGspace$. Since $u$ is absolutely continuous, there exist $t_0\in I$ such that
	\begin{equation}
	\label{eq:u_meanval}
	\begin{split}
	&u(t_0)=\frac{1}{\abs{I}}\int_{I}u\,dt\quad\text{ and}\\ 
	&u(t)-u(t_0)=\int_{t_0}^{t}\dot u \,dt\quad\text{for }t\in I.
	\end{split}
	\end{equation}
	\[\]
	
	Hence, by convexity of $G$ and Jensen's integral inequality we obtain 
	\begin{multline*}
	\int_IG(\dot u)\,dt+\int_I G(u)\,dt \geq\frac{|I|}{|t-t_0|}\int_{t_0}^tG\left(\frac{|t-t_0|}{|I|}\dot u\right)\,dt+\int_IG(u)\,dt\geq\\\geq |I|G\left(\frac{1}{|I|}\int_{t_0}^t\dot u\right)+|I|G\left(\frac{1}{\abs{I}}\int_{I}u\right)=\\=|I|\left(G\left(\frac{u(t)-u(t_0)}{|I|}\right)+G\left(u(t_0)\right)\right).
	\end{multline*}
	
	Since $\abs{I}\geq 1$ and G is convex, we have
	\begin{multline*}
	R_G(\dot u)+R_G(u)\geq G\left(\frac{u(t)-u(t_0)}{|I|}\right)+G\left(u(t_0)\right)\geq\\\geq G\left(\frac{u(t)-u(t_0)}{\abs{I}}\right)+G\left(\frac{u(t_0)}{\abs{I}}\right)\geq  2 G\left(\frac{u(t)}{2\abs{I}}\right).
	\end{multline*}
\end{proof}
In the proof of the existence of a second solution we will need the following inequality by Brezis and Lieb \cite[lem. 3]{BreLie83} (see also \cite[lem. 4.7]{KhaKoz15}).
\begin{lemma}
	\label{lem:BreLie}
	For all $k>1$, $0<\epsilon<\frac1k$, $x,y\in\R^N$
	\[
	|G(x+y)-G(x)|\leq \epsilon|G(kx)-kG(x)|+2G(C_\epsilon y),
	\]
	where $C_\epsilon=\frac{1}{\epsilon(k-1)}$.
\end{lemma}

Let us also recall proposition 2.4 from \cite{ChmMak19}, which  will be used in the proof of the fact that our action functional satisfies the Palais-Smale condition. 
\begin{proposition}
	\label{prop:adelkowyTrik}
	For any $1<p\leq q<\infty$, such that $|\cdot|^p\prec G(\cdot)\prec |\cdot|^q$,
	\[
	\begin{split}
	\int_{I}|u|^{p}\,dt\geq C\,\WLGnorm{u}^{p-q}\LGnorm{u}^{q}
	\end{split}
	\]
	for  $u\in\WLGspace\backslash\{0\}$ and some $C>0$.
\end{proposition}
\section{Existence of the first solution}
Define the action functional $\Jcal\colon \WLGTspace(I,\R^N)\to \R$ by
\begin{equation}
\label{eq:J}\tag{$\Jcal$}
\Jcal(u)=\int_I F(t,u,\dot{u}) + V(t,u) + \inner{f}{u}\,dt.
\end{equation}
Since $F$ and $V$ are of class $C^1$ and $F$ satisfies \ref{asm:F:growth}, $\Jcal$ is well defined and of class $C^1$. Furthermore, its derivative is given by
\begin{equation}
\label{eq:J'}\tag{$\Jcal'$}
\Jcal'(u)\varphi =\int_I \inner{F_v(t,u,\dot{u})}{\dot{\varphi}}\,dt+
\int_I \inner{F_x(t,u,\dot{u})+V_x (t,u)+f}{\varphi}\,dt
\end{equation}
See \cite[Theorem 5.7]{ChmMak17} for more details.

We can now formulate the first of our main results.
\begin{theorem}
	\label{thm:main1}
	Let $F$, $V$ and $f$  satisfies  \ref{asm:F:convex}-\ref{asm:F:zero}, \ref{V:K-W}-\ref{V:zero} and  
	\ref{f:IntG*f<}.
	Then \eqref{eq:AELT} has at least one nontrivial periodic solution.
\end{theorem}
\begin{proof} We show that $\Jcal$ satisfies assumptions of theorem \ref{thm:MPT_mw} for 
	$\Omega=\Phi^{-1}([0,\rho))$, where $\rho>0$, $\Phi:\WLGspace\to[0,\infty)$, $\Phi(u)=R_G(\dot u)+R_G(u)$.
	\begin{description} 
		\item[Step 1]
		One can see, that $\Omega$ is a bounded open neighborhood of $0$. We claim that $\partial\Omega=\Phi^{-1}(\{\rho\})$. Since $\Phi$ is continuous, \[\partial\Phi^{-1}([0,\rho))\subset\Phi^{-1}(\partial[0,\rho))=\Phi^{-1}(\{\rho\}).\] For the opposite inclusion (not true in general) suppose that $x\in\Phi^{-1}(\{\rho\})$, $x_n^1=\frac{n+1}{n}x$, $x_n^2=\frac{n}{n+1}x$ for $n\in\N.$ Then $x_n^1,x_n^2\to x$. From convexity of $\Phi$ and since $\rho>0$, we have \[\Phi(x_n^1)\geq\frac{n+1}{n}\Phi(x)>\Phi(x) \quad\text{ and }\quad \Phi(x_n^2)\leq\frac{n}{n+1}\Phi(x)<\Phi(x)\] for all $n\in\N$.  Hence $x\in\partial\Omega.$
		\item[Step 2] In this step we show that  $\Jcal$  satisfies the Palais-Smale condition.
		
		Fix $u\in\WLGTspace$, such that $\WLGnorm{u}\neq 0$.
		From \ref{asm:F:AR} we have 
		\begin{equation*}
		\label{PS:FV}
		\int_I\inner{F_v(t,u,\dot u)}{\dot u}+\inner{F_x(t,u,\dot u)}{ u}\,dt
		\leq \theta_F\int_I F(t,u,\dot u)\,dt.
		\end{equation*}
		
		Combining it with \ref{asm:F:ellipticity} we have:
		\begin{equation}\label{PS:F}
		\int_{I}\theta_V F(t,u,\dot u)-\inner{F_v(t,u,\dot u)}{\dot u}-\inner{F_x(t,u,\dot u)}{u}\,dt\geq \Lambda(\theta_V-\theta_F)\int_IG(\dot u)\,dt.
		\end{equation}

		Set
		\[ M_V=\sup\{| (\theta_V-\varepsilon_V)K(t,x)-\theta_VW(t,x))-\inner{V_x(t,x)}{x}| \colon t\in I, |x|\leq M \}.\]

		Then, by assumptions \ref{V:K-W} and \ref{V:ARad},
		\begin{multline}\label{PS:V}
		\theta_V\int_I V(t,u)-\inner{ V_x(t,u)}{u}\,dt
		=\varepsilon_V\int_IK(t,u)\,dt+\\
		+(\theta_V-\varepsilon_V)\int_IK(t,u)\,dt
		-\theta_V\int_I W(t,u)\,dt-\int_I\inner{V_x(t,u)}{u}\,dt \geq\\
		\geq\varepsilon_V\int_IK(t,u)\,dt-|I|M_V.
		\end{multline}
		
		Applying \eqref{V:W>K>pk} 
		and proposition \ref{prop:adelkowyTrik} we have that for some $C_1, C_2=C_2(|I|)>0$ 
		\begin{equation}\label{PS:K}
		\int_{I}K(t,u)\,dt\geq C_1\frac{\LGnorm{u}^{q}}{\WLGnorm{u}^{q-p_K}}-C_2
		\end{equation}
		for all $q$ such that $G\prec|\cdot|^{q}$.
		
		Since $f\in\Lpspace[G^{\ast}]$, by H\"older's inequality  we have
		\begin{equation}
		\label{PS:f_below}
		(\theta_V-1) \int_I \inner{f(t)}{u} \,dt
		\geq
		-2(\theta_V-1)\LGastnorm{f} \LGnorm{u}
		\geq
		-  C_f \WLGnorm{u},
		\end{equation}
		where $C_f=2 (\theta_V-1)\LGastnorm{f}>0$.

		Let $\{u_n\}\subset\WLGspace$ be a Palais-Smale sequence for functional $\Jcal$. Then there exist $C_J, C_{J'}>0$ such that
		
		\begin{equation}
		\label{eq:CJJ'}
		-C_J\leq\Jcal(u_n)\leq C_J,\quad -C_{J'}\WLGnorm{u_n}\leq\Jcal'(u_n)u_n\leq C_{J'}\WLGnorm{u_n}.
		\end{equation}
		
		Assume that $\{u_n\}$ is not bounded. Then there exists a subsequence of $\{u_n\}$ such that $\WLGnorm{u_n}\to\infty$.
		
		Combining \eqref{PS:F}-\eqref{eq:CJJ'}		
		we obtain
		\begin{multline*}
		\label{PS:ogr_beg}
		\theta_V C_J+C_{J'}\WLGnorm{u_n}\geq \theta_V\Jcal(u_n)-\Jcal'(u_n)u_n=\\=\int_{I}\theta_V F(t,u_n,\dot u_n)-\inner{F_v(t,u_n,\dot u_n)}{\dot u_n}-\inner{F_x(t,u_n,\dot u_n)}{u_n}\,dt+\\
		+\int_I\theta_V V(t,u_n)-\inner{ V_x(t,u_n)}{u_n}\,dt+\int_I(\theta_V-1)\inner{f(t)}{u_n}\,dt\geq\\
		\geq C_0\int_IG(\dot u_n)\,dt +C_1\frac{\LGnorm{u_n}^{q}}{\WLGnorm{u_n}^{q-p_K}}-C_2-  C_3 \WLGnorm{u_n}
		\end{multline*}
		for some $C_0,C_1,C_2,C_3>0$. Hence
		\[
		\WLGnorm{u_n}\left(\frac{R_G(\dot u_n)}{\WLGnorm{u_n}}+\frac{\LGnorm{u_n}^{q}}{\WLGnorm{u_n}^{q-p_K+1}}-\frac{C_4}{\WLGnorm{u_n}}-  C_5\right)\leq C_6\]
		for some $C_4,C_5,C_6>0$.
		
		In the proof of lemma 3.2 in \cite{ChmMak19} it was shown that the left side of the above inequality goes to the infinity, which is impossible. Hence $\{u_n\}$ is bounded.

		Repeating arguments used in the proof of lemma 4.2 from \cite{ChmMak20} one can show that there exists a converging subsequence i.e. the  Palais Smale condition is met.
		\item[Step 3]
		%
		Take $u\in \Phi^{-1}(\{\rho\})$.
		Then, by lemma \ref{lem:RG>G}, we have
		$G\left(\frac{u(t)}{2|I|}\right)\leq\rho/2$ for all $t\in I$.
		From \ref{asm:F:ellipticity}, \ref{V:>G} and Fenchel's inequality,
		\begin{multline*}
		\Jcal(u)\geq \int_I \Lambda G(\dot u)+bG(u)-G(u)-G^{\ast}(f)-g(t)\,dt\geq
		\\\geq \min\{\Lambda,b-1\}\rho-\int_I\left(G^{\ast}(f)+g(t)\right)\,dt.
		\end{multline*}
		Combining it with \ref{f:IntG*f<} we obtain $\Jcal(u)>0$ on $\Phi^{-1}(\{\rho\})$. 
		\item[Step 4]
		Now we show that there exists $e\in\Phi^{-1}((\rho,\infty))$ such that $\Jcal(e)<0$.
		By \ref{V:ARad}, for $|x|>M$, $t\in I$, $\lambda>1$ we obtain
		\begin{multline*}
		\log\left(\frac{-V(t,\lambda x)}{-V(t,x)}\right) =
		\int_{1}^{\lambda}\frac{d}{d\lambda}\log(-V(t,\lambda x))\,d\lambda=\int_1^{\lambda}\frac{-\inner{ V_x(t,\lambda x)}{\lambda x}}{-\lambda V(t,\lambda x)}\,d\lambda\geq\\
		\geq \int_1^{\lambda}\frac{-(\theta_V-\varepsilon_V) K(t,\lambda x)+\theta_V	W(t,\lambda x) }{-\lambda V(t,\lambda x)}\,d\lambda\geq \theta_V\int_1^{\lambda}\frac{1}{\lambda}=\log\left(\lambda^{\theta_V}\right).
		\end{multline*}
		
		Hence
		\[V(t,\lambda x)\leq \lambda^{\theta_V}V(t,x)\text{ for } |x|>M.\]
		In similar way,  from \ref{asm:F:AR}, we have:
		\[F(t,\lambda x,\lambda v)\leq \lambda^{\theta_F}F(t,x,v)\] for $x,v\in\R^N$, $t\in I$, $\lambda>1$.
		Let $\lambda>1$ and	 $\psi\in\WLGTspace$ be such that $\left|\{t\in I\colon~|\psi(t)|>0\}\right|>0$.
		Then we obtain
		\begin{multline*}
		\label{Je<03zmienne
		}
		\Jcal(\lambda\psi) 
		\leq \int_I\lambda^{\theta_F} F(t,\psi,\dot\psi)+\lambda\inner{f(t)}{\psi}\,dt+\lambda^{\theta_V}\int_{\{|\psi(t)|>M\}}V(t,\ \psi)\,dt+C_V|I|,
		\end{multline*}
		where 	
		$C_V=\sup\{V(t,x):\, |x|<M,t\in I\}$.
		Note that $V$ is negative for $|x|>M$ and  $\theta_V>\theta_F$. Therefore if we take $e=\lambda\psi$ for sufficiently large $\lambda$,  we get $\Jcal(e)<0$ and $\Phi(e)>\rho$.
		\item[Step 5]
		To finish the proof note that by \eqref{eq:J}, \ref{asm:F:zero} and \ref{V:zero} we have that $\Jcal(0)=0$. 
		Applying theorem \ref{thm:MPT_mw} to $\Jcal$, $e_0=0$ and $e_1=e$, we obtain that there exists a critical point $u_1$ with a critical value
		\begin{equation}
		\label{c1}\tag{$c_1$}
		c_1:=\inf_{g\in\Gamma}\max_{s\in[0,1]}\Jcal(g(s))>0,\end{equation}
		where
		\[
		\Gamma=\left\{g\in C\left([0,1],\WLGspace\right)\::\:g(0)=0,~~g(1)=e\right\}.
		\]
	\end{description}
\end{proof}

\section{Existence of the second solution}

Theorem \ref{thm:main1} ensures the existence of the first solution of \eqref{eq:AELT}.
To obtain the second nontrivial solution we use the well known Ekeland's Variational Principle.
\begin{theorem}[thm. 1.1, \cite{Eke74}]
	\label{thm:wze}
	Let $V$ be a complete metric space  and $J:V\to\R\cup\{+\infty\}$  a lower semi continuous function, bounded from below, $\not\equiv+\infty$.
	Let $u\in V$ and $\varepsilon>0$ be such that
	\begin{equation}
	\label{eq:ekel_zal}
	J(u)\leq \inf_{v\in V}J(v)+\varepsilon.
	\end{equation} 
	Then for all
	$\delta>0$
	there exists some point $v\in V$ such that
	\begin{itemize}
		\item[i)]	$J(v)\leq J(u),$
		\item[ii)]	$d(u,v)\leq\delta,$
		\item[iii)] $J(w)>J
		(v)-\frac{\varepsilon}{\delta} d(v,w)$ for all $w\neq v$.
	\end{itemize}
\end{theorem} 
Set
\begin{equation}
\label{c2}
\tag{$c_2$}
c_2:=\inf_{u\in\overline\Omega}\Jcal(u)
\end{equation}
Let us recall, that
$\Omega=\Phi^{-1}([0,\rho)).$
Firstly we consider the case with forcing.
\begin{theorem}
	\label{thm:main3}
	Let $F$ and $V$  satisfies  \ref{asm:F:convex}-\ref{asm:F:zero}, \ref{V:K-W}-\ref{V:zero} and  $f(t)\not\equiv 0$.
	Then \eqref{eq:AELT} has at least two  periodic solutions.
\end{theorem}
\begin{proof}
	Note that $\overline\Omega$ is a complete metric space  with respect to the norm in $\WLGspace$ and $\Jcal$ is bounded from below on $\overline{\Omega}$. Fix $\varepsilon>0$ and choose $\delta=\sqrt{\varepsilon}$.
	There exists $u\in\overline{\Omega}$ such that
	$\Jcal(u)\leq c_2
	+\varepsilon.$
	By theorem \ref{thm:wze}, there exists $v\in\overline{\Omega}$ such that 
	\begin{equation}
	\label{eq:WZA1zas}
	c_2\leq\Jcal(v)\leq c_2+\varepsilon,
	\end{equation}
	\begin{equation}
	\label{eq:WZAzas}
	\WLGnorm{v-u}\leq\sqrt{\varepsilon},
	\end{equation}
	\begin{equation}
	\label{eq:WZA2zas}
	\Jcal(w)\geq \Jcal(v)-\sqrt{\varepsilon}\WLGnorm{w-v} \text{ for all }w\neq v.
	\end{equation} 
	
	Now we show that $v\in\Omega$.  Since $\Jcal(0)=0$, $c_2\leq 0$. Hence and by \eqref{eq:WZA1zas} we have that
	\[
	0\geq c_2\geq \Jcal(v)-\varepsilon
	\]
	If we assume that $v\in\partial\Omega$, then $\Jcal(v)>0$, by Step 2 in the proof of theorem \ref{thm:main1}.
	Taking sufficiently small $\varepsilon$, we deduce that $0\geq c_2\geq \Jcal(v)-\varepsilon>0$, which is a contradiction.
	
	Take $w=v+sh$ with $0<s\leq 1$, $h\in\WLGspace$ such that $\WLGnorm{h}=1$. Then, by lemma \ref{lem:BreLie} we have that 
	\[
	\int_IG(v+sh)\,dt\leq \int_IG(v)+\sqrt{s}|G(2v)-2G(v)|+2G(\sqrt{s}h)\,dt
	\]
	and
	\[
	\int_IG(\dot v+s\dot h)\,dt\leq \int_IG(\dot v)+\sqrt{s}|G(2\dot v)-2G(\dot v)|+2G(\sqrt{s}\dot h)\,dt.
	\]
	Hence
	\begin{multline*}
	\Phi(v+sh)\leq\int_{I} G(v)+G(\dot v)+\sqrt{s}\left(|G(2v)-2G(v)|+|G(2\dot v)-2G(\dot v)|\right)\,dt+\\+2\int_IG(\sqrt{s}h)+G(\sqrt{s}\dot h)\,dt.
	\end{multline*}
 Note that $\LGnorm{\sqrt{s}h}\leq1$. From \eqref{eq:modular_norm_relation} we obtain $\int_IG(\sqrt{s}h)\leq  \LGnorm{\sqrt{s}h}\leq\sqrt{s}$. Hence
 \[
\Phi(v+sh)\leq\int_{I} G(v)+G(\dot v)+\sqrt{s}\left(|G(2v)-2G(v)|+|G(2\dot v)-2G(\dot v)|\right)\,dt+4\sqrt{s}.\]
	Since  $\Phi(v)<\rho$, it follows that for $s$ sufficiently small
 $\Phi(v+sh)<\rho$.
	By \eqref{eq:WZA2zas} 
	\[
	\Jcal(v+sh)\geq\Jcal(v)-\sqrt{\varepsilon}\WLGnorm{sh}.
	\]
	Hence
	\[
	\frac{\Jcal(v+sh)-\Jcal(v)}{s}\geq -\sqrt{\varepsilon}.
	\]
	Taking the limit as $s\to 0$, we have 
	$
	\inner{\Jcal'(v)}{h}\geq -\sqrt{\varepsilon}
	$
	for  $h\in\WLGspace$ such that $\WLGnorm{h}=1$.
	Since $-h\in\Omega$, we have
	$
	\sup_{\norm{h}=1}|\inner{\Jcal'(v)}{h}|\leq \sqrt{\varepsilon}$
	and hence
	\begin{equation}
	\label{eq:J'<eps}
	\WLGastnorm{\Jcal'(v)}\leq\sqrt{\varepsilon}.
	\end{equation}
	
	Let $\{u_n\}$ be a minimizing sequence of $\Jcal$. 
	We choose $\varepsilon_n$ in the following way:
	\[\varepsilon_n=
	\begin{cases}\Jcal(u_n)-\inf_{\overline{\Omega}}\Jcal,&\Jcal(u_n)>\inf_{\overline{\Omega}}\Jcal\\ \frac1n,&\Jcal(u_n)=\inf_{\overline{\Omega}}\Jcal\end{cases}
	\]
	One can see, that $\varepsilon_n\to0$ as $n\to\infty$. 
	Since $u_n$ satisfies \eqref{eq:ekel_zal} for each $n$, we have
	\begin{equation}
	\label{eq:WZA1}
	c_2\leq\Jcal(v_n)\leq c_2+\varepsilon_n,
	\end{equation}
	\begin{equation}
	\WLGnorm{v_n-u_n}\leq \sqrt{\varepsilon_n},
	\end{equation}
	\begin{equation}
	\WLGastnorm{\Jcal'(v_n)}\leq\sqrt{\varepsilon_n}.
	\end{equation} 
	for $v_n$ associated to $u_n$ and $\varepsilon_n$ 
	in \eqref{eq:WZA1zas}-\eqref{eq:WZA2zas}. Hence we can see that $\{v_n\}$ is a Palais-Smale sequence of $\Jcal$.
	Since $\Jcal$ satisfies the Palais-Smale condition (step 2 in the proof of theorem \ref{thm:main1}), we have that there is $u_2$ such that $\Jcal(u_2)=c_2$ and $\Jcal'(u_2)=0$. 	By \ref{asm:F:zero} and $f(t)\not\equiv 0$ we have that $u_0
	\equiv 0$ is not a solution of \eqref{eq:AELT}, so $u_2$  is a nontrivial solution of \eqref{eq:AELT}.

\end{proof}
\begin{remark}
	Similar arguments regarding the existence of a second solution can  be found in \cite[the proof of lem. 3.3]{Jan12} and \cite[step 4 in the proof of thm. 1.1]{YuaZha16}. See also \cite[thm. 4.2, cor. 4.1]{MawWil89}, \cite[thm. 3.1]{AmbHaj18} for more detailed computations.
\end{remark}

If $f(t)\equiv 0$ it is necessary to show that $\inf_{u\in\overline\Omega}\Jcal(u)<0$.
Without this assumption it is possible that the minimizing sequence
found in the proof of theorem \ref{thm:main3} 
converges to the  solution $u_0\equiv0$. In order to avoid this phenomenon
we add assumptions \ref{asm:F:nabla2type} and \ref{V:nabla2type}. 
\begin{theorem}
	\label{thm:main2}
	Let $F$ and $V$  satisfies  \ref{asm:F:convex}-\ref{asm:F:nabla2type}, \ref{V:K-W}-\ref{V:nabla2type} and  $f(t)\equiv 0$.
	Then \eqref{eq:AELT} has at least two nontrivial periodic solutions.
\end{theorem}

\begin{proof}
	Let  $0\neq\psi\in\Omega$.  
	Then, by lemma \ref{lem:RG>G},  $G\left(\frac{\psi(t)}{2|I|}\right)\leq\rho/2$ for all $t\in I$. Choose \[\lambda<\min\left\{1,\left(\frac{C_2}{C_1}\right)^{1/(\min\{\zeta_F,\zeta_K\}-\zeta_W)}\right\}.\]
	By \ref{asm:F:nabla2type},  \ref{V:nabla2type} and $f(t)\equiv0$ we have
	\begin{multline*}
	\Jcal(\lambda\psi)= \int_IF(t,\lambda\psi,\lambda\dot\psi)+K(t,\lambda\psi)-W(t,\lambda\psi)\,dt\leq\\
	\leq\int_I\lambda^{\zeta_F}F(t,\psi,\dot\psi)+\lambda^{\zeta_K}K(t,\psi)-\lambda^{\zeta_W}W(t,\psi)\,dt\leq\\\leq \int_IC_1\lambda^{\min\{\zeta_F,\zeta_K\}}-C_2\lambda^{\zeta_W}\,dt<0,
	\end{multline*}
	where $C_1=\max_{u\in\Omega}F(t,u,\dot u)+K(t,u)$, $C_2=\min_{u\in\Omega}W(t,u)$.
	
	Hence $c_2<0$. In the same way as in the proof of theorem \ref{thm:main3} we show that there exists a minimizing sequence $\{v_n\}$ such that $\Jcal(v_n)\to 0$ and $\Jcal'(v_n)\to c_2$ as $n\to\infty$.  Hence there is $u_2$ such that $\Jcal(u_2)=c_2$ and $\Jcal'(u_2)=0$ and $u_2$ is the second solution of \eqref{eq:AELT}. Since $c_2<0$ and $\Jcal(0)=0$,  $u_2$ is nontrivial. 
	
\end{proof}

\section{Example }
\label{sec:example}
We finish the paper with some example of function $F$, potential and forcing satisfying assumptions \ref{V:K-W}-\ref{V:nabla2type}, \ref{asm:F:convex}-\ref{asm:F:nabla2type} and \ref{f:IntG*f<}. We also show that they do not satisfy assumptions from \cite{ChmMak19}, which indicate that taking $\Phi^{-1}([0,\rho))$ instead of the ball $\{u\in\WLGspace\colon \WLGnorm{u}<r\}$ is better to obtain the mountain pass geometry in the anisotropic case.

Let us recall that in \cite{ChmMak19} 
condition \eqref{asm:mpt_ar_J>0}, mentioned in the introduction, was
guaranteed by the following assumptions:

\begin{enumerate}[label=($A_3$)]
	\item\label{V>Gr0} 	there exist $r_0$, $b>1$ and $a\in\LGspace[1](I,\R)$, such that $V(t,x)\geq b\, G(x)-a(t) $ for 
	$|x|\leq r_0$,
\end{enumerate}

and either
\begin{equation}
\label{asm:f_equinorm}
\tag{$f_1$}
\quad
\int_{-T}^{T}G^{\ast}(f(t))+a(t)\,dt
<
\min\{1,b-1\}(r_0/(2C_{\infty,G}))
\end{equation}
or 
\begin{equation}
\label{asm:f_simonenko}
\tag{$f_2$}
\int_{-T}^{T}G^{\ast}(f(t)) +  a(t)\,dt
<\min\{1,b-1\}\begin{cases}	(r_0/(2C_{\infty,G}))^{q_G},&r_0\leq 4C_{\infty,G}\\(r_0/(2C_{\infty,G}))^{p_G}&r_0> 4C_{\infty,G}\end{cases},
\end{equation} 
where $C_{\infty,G}$ is an embedding  constant for $\WLGspace\embed \Lpspace[\infty]$  given by formula  $C_{\infty,G}= \max\{1,|I|\} A_G^{-1}\left(\frac{1}{|I|}\right)$ and $A_G\colon [0,\infty)\to [0,\infty)$ is the greatest convex minorant of $G$ (see \cite{AciMaz17}), $r=r_0/C_{\infty, G}.$ 


We show  that sets of assumptions
\begin{enumerate}[label=(\roman*)]
	\item \ref{V>Gr0}, (\ref{asm:f_equinorm} or \ref{asm:f_simonenko}),
	\item \ref{V:>G}, \ref{f:IntG*f<}
\end{enumerate}  
are independent, namely for some potentials it is not possible to find $r_0$ such that the first assumptions are satisfied, but 
for the same potential one can find $\rho$ such that the latter are met.

\begin{example}\label{ex:first}
	Let $I=[-1,1],$ $x=(x_1,x_2), v=(v_1,v_2)\in\R^N\times\R^N$
	\begin{align*}
	&F(t,x,v)=G(v)=v_1^2+(v_1-v_2)^2,\\
	&K(t,x)=2G(x)+|x|^2\log(|x|^2+1),\\
	&W(t,x)=\left(G(x)\right)^2+\frac{|x|^{\frac32}+|x|^5}{100},\\
	&V(t,x)=K(t,x)-W(t,x),\\
	&f_0(t)=\frac{2-t^2}{2500},\, f=(f_0,...,f_0),
	\end{align*}
	\begin{align*}
	&g(t)\equiv 0.001,\quad b=2, \quad \rho=0.004, \\
	&\theta_V=4.9,\quad\varepsilon_V=0.001,\quad\theta_{F}=4,\\
	&\zeta_W=\frac{31}{16},\quad\zeta_K=2,\quad\zeta_{F}=2.
	\end{align*}

	If we take $F$, $V$ and $f$ given above, then there do not exist $a$, $b$, $r_0$ such that either  \ref{V>Gr0} and \eqref{asm:f_equinorm} or \ref{V>Gr0} and  \eqref{asm:f_simonenko} holds.

	\begin{figure}[h]
		\centering
		\includegraphics[width=0.5\linewidth]{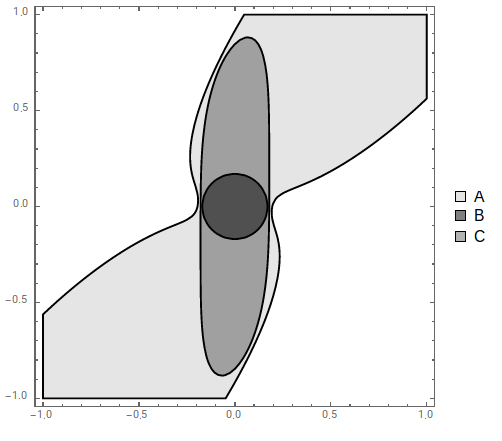}
		\caption{ The shape of the area  $A=\left\{x\in\R^N\colon V(x)\geq 
			bG(x)-g\right\}$ is such that for a ball $B=B_{r_0}(0)$ with  $r_0$ such that $B\subset A$ neither \eqref{asm:f_equinorm} nor \eqref{asm:f_simonenko} is satisfied.
			For area $C=\left\{x\in \R^N\colon G\left(\frac{x}{2|I|}\right)\leq\rho/2\right\}\subset A$ condition \ref{f:IntG*f<} is satisfied.}
		\label{fig:zalozenia}
	\end{figure}
	This situation was shown in Figure \ref{fig:zalozenia}.
	The shape of the area
	\[A=\left\{x\in\R^N\colon V(x)\geq 
	bG(x)-g\right\}\] is such that for a ball \[B=\{x\in\R^N\colon |x|\leq r_0\}\] with sufficiently small radius (such that $B
	\subset A$) neither \eqref{asm:f_equinorm} nor \eqref{asm:f_simonenko} is satisfied.
	\begin{figure}[!htb]
		\begin{minipage}{0.48\textwidth}
			\centering
			\includegraphics[width=.9\linewidth]{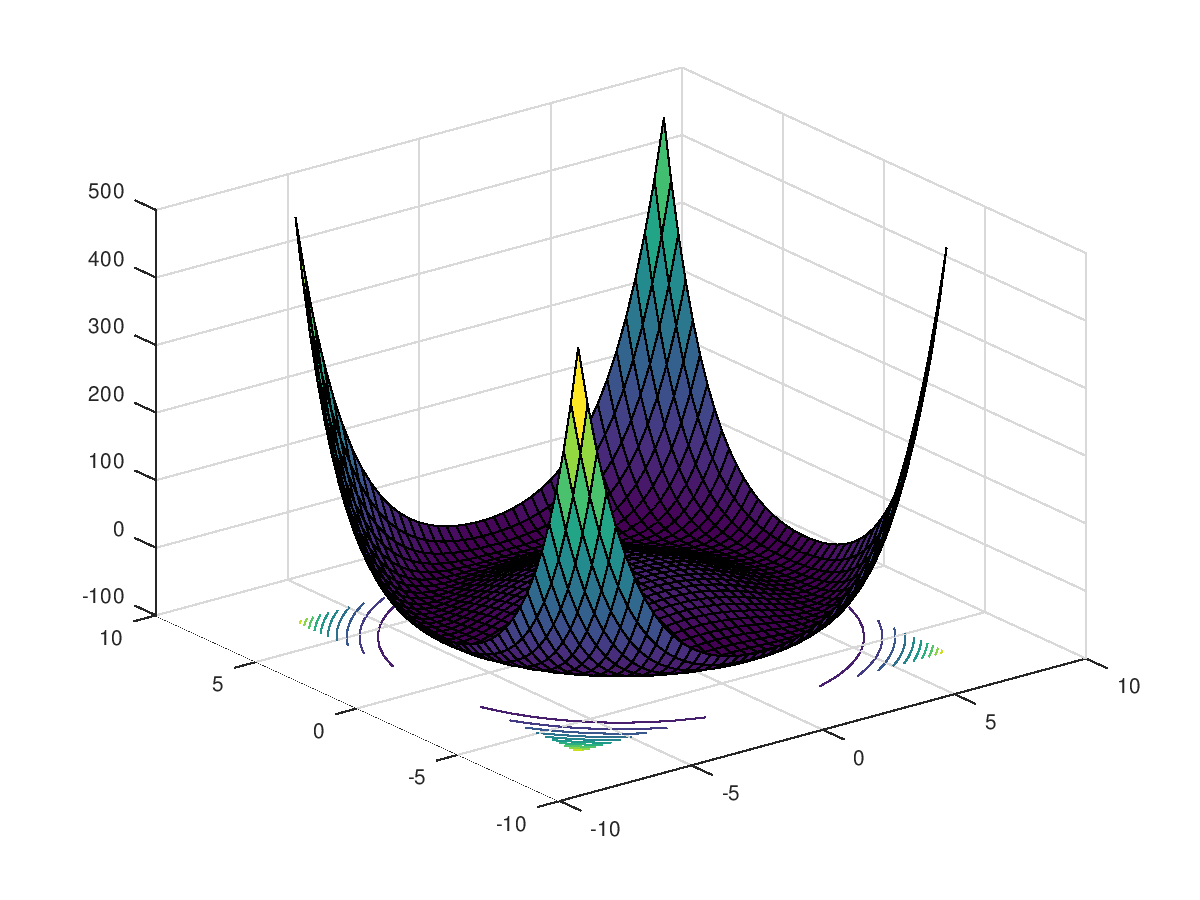}
			\caption{Plot of the function $h_1$}\label{Fig:plot1}
		\end{minipage}\hfill
		\begin{minipage}{0.48\textwidth}
			\centering
			\includegraphics[width=.9\linewidth]{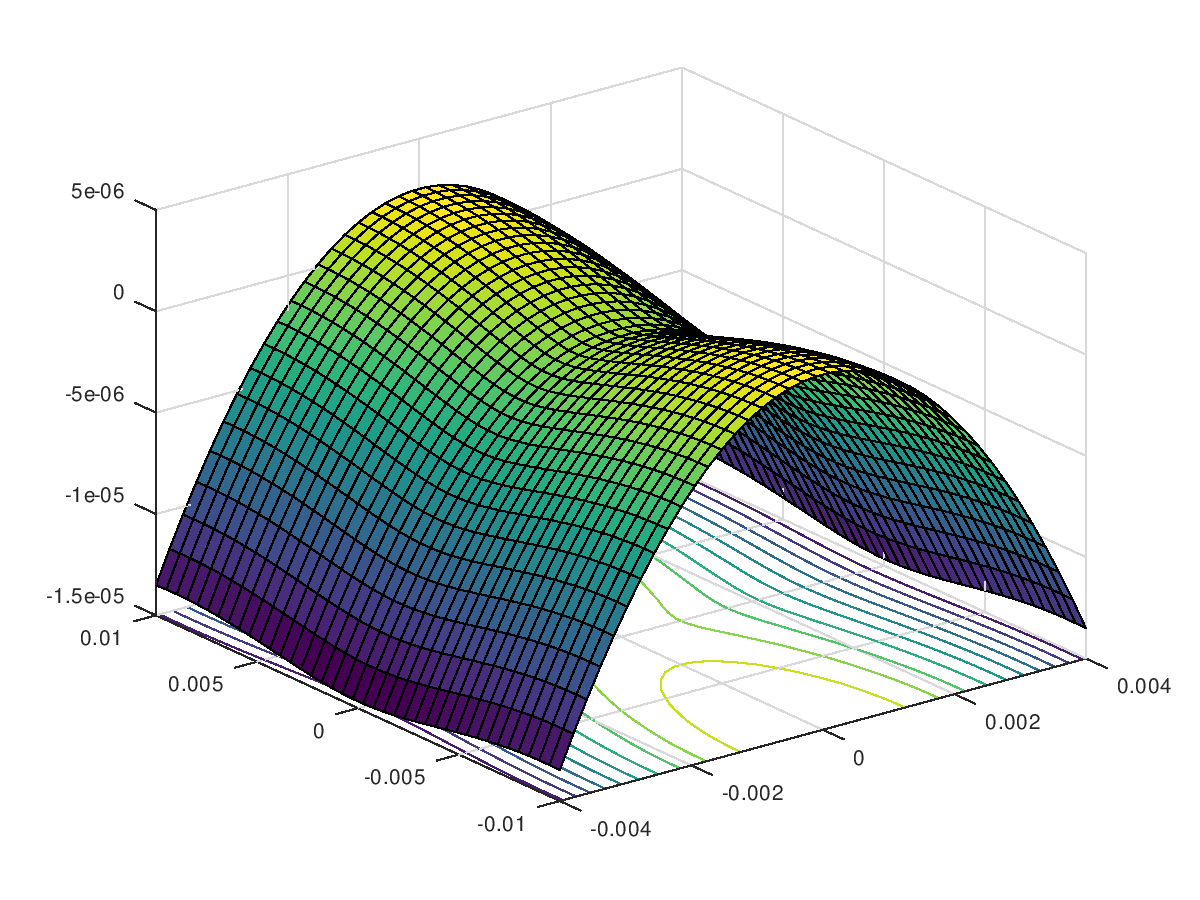}
			\caption{Plot of the function $h_2$}\label{Fig:plot2}
		\end{minipage}
	\end{figure}
	
	\begin{itemize}
		\item  If we assume that \eqref{asm:f_equinorm} is met, then we  can check, that for each $r_0>4C_{G,\infty}$ there exists $x\in B_{
			r_0}(0)$ such that $V(x)<bG(x)-a$ for all $b>1$, $a<\frac{r_0}{2C_{\infty,G}}$, which contradicts assumption $(A_3)$. In fact, if we take $x\in\partial B_{
			r_0}(0)$, then  it suffices to show that function $h_1:\R^N\to\R$,  $h_{1}(x)= G(x)-V(x)-\frac{|x|}{2C_{G,\infty}} $ can have  positive values for any $r_0>4C_{G,\infty}$ (see  Figure \ref{Fig:plot1}).
		\item If we assume that \eqref{asm:f_simonenko} is met, then we can check that for each $r_0>0$ there exists $x\in B_{
			r_0}(0)$ such that $V(x)<bG(x)-a$ for all $b>1$, $a<\min\left\{(\frac{r_0}{2C_{\infty,G}})^2, 
		(\frac{r_0}{2C_{\infty,G}})^4\right\}$ .
		So it suffices to show that function $h_2:\R^N\to \R$,  $h_{2}(x)= G(x)-V(x)-\left(\frac{|x|}{2C_{G,\infty}}\right)^2 $ can have  positive values for any $r_0$ (see Figure \ref{Fig:plot2}).
		
	\end{itemize}

	If, instead of the ball, we take "more anisotropic" area  \[C=\left\{x\in \R^N\colon G\left(\frac{x}{2|I|}\right)\leq\rho/2\right\}\subset A,\] connected with condition \ref{V:>G}, then  \ref{f:IntG*f<} is satisfied.
	%
	%


\end{example}

\begin{remark}
	In the above example we can also take 
	a more complicated $F$, e.g.\[
	F(t,x,v)=G(v)(2+|x|^{9/2}-\sin{t}).	\]
\end{remark}

\bibliographystyle{elsarticle-num}
\bibliography{multiplicity}
\end{document}